\documentclass{amsart}
\usepackage{amsmath,amsxtra,amsthm,amssymb,amsfonts,graphicx,color,soul}
\usepackage[pdfusetitle, bookmarks=true,bookmarksnumbered=false,bookmarksopen=false, 
  breaklinks=false,pdfborder={0 0 0},backref=false,colorlinks=true]{hyperref}

\newcommand{\red}{\textcolor{black}}
\newcommand{\commentred}[1]{}

\newcommand{\commentblue}[1]{}

\newcommand{\commentgreen}[1]{}

\newcommand{\commentmag}[1]{}

\newcommand{\CAT}{\mathrm{CAT}}
\newcommand{\CBB}{\mathrm{CBB}}

\newcommand{\R}{\mathbb{R}}

\newcommand{\N}{\mathbb{N}}

\newcommand{\bel}[1]{\begin{equation}\label{#1}}

\newcommand{\be}{\begin{equation}}
\newcommand{\ee}{\end{equation}}
\newcommand{\ba}{\begin{eqnarray}}
\newcommand{\ea}{\end{eqnarray}}
\newcommand{\rf}[1]{(\ref{#1})}

\newcommand{\qe}{\end{equation}}

\newtheorem{theorem}{Theorem}[section]

\newtheorem{lemma}[theorem]{Lemma}
\newtheorem{corollary}[theorem]{Corollary}
\newtheorem{definition}[theorem]{Definition}

\newtheorem{proposition}[theorem]{Proposition}
\newtheorem{remark}[theorem]{Remark}

\def\nat{\mathbb{N}}                
\def\rls{\mathbb{R}}                
\def\hs{\mathcal{H}}                
\def\model{\mathbb{M}_\kappa^2}      

\def\eps{\varepsilon}
\def\ol{\overline}
\def\l{\left}
\def\r{\right}
\def\col{\colon}
\def\res{\restriction}

\def\length{\operatorname{length}}
\def\spn{\operatorname{span}}
\def\sec{\operatorname{Sec}}
\def\curv{\operatorname{Curv}}

\def\di{\operatorname{d}\!}

\def\as{\!\mathrel{\mathop:}=} 

\begin{document}

\title{\red{A notion of nonpositive curvature for general metric spaces}}
\author[M. Ba\v{c}\'ak \and B. Hua \and J. Jost \and M. Kell \and A. Schikorra]{Miroslav Ba\v{c}\'ak \and Bobo Hua \and J\"{u}rgen Jost \and Martin Kell \and Armin Schikorra}
\date{\today}
\thanks{The research leading to these results has received funding from the
 European Research Council under the European Union's Seventh Framework
 Programme (FP7/2007-2013) / ERC grant agreement no 267087.}


\address{Max Planck Institute for Mathematics in the Sciences, Inselstr.~22, 04103 Leipzig, Germany}
\curraddr[B. Hua]{School of Mathematical Sciences, LMNS, Fudan University, Shanghai 200433, China}
\email[M. Ba\v{c}\'ak]{miroslav.bacak@mis.mpg.de}
\email[B. Hua]{bobohua@fudan.edu.cn}
\email[J. Jost]{jost@mis.mpg.de}
\email[M. Kell]{mkell@mis.mpg.de}
\email[A. Schikorra]{armin.schikorra@mis.mpg.de}

\begin{abstract}
We introduce a new definition of nonpositive curvature in metric spaces and study its relationship to the existing notions of nonpositive curvature in comparison geometry. The main feature of our definition is that it applies to all metric spaces and does not rely on geodesics. Moreover, a scaled and a relaxed version of our definition are appropriate in discrete metric spaces, and are believed to be of interest in geometric data analysis.
\end{abstract}

\maketitle


\section{Introduction}

The aim of the present paper is to introduce a new definition of nonpositive curvature in metric spaces. Similarly to the definitions of Busemann and CAT(0) spaces, it is based on comparing triangles in the metric space in question with triangles in the Euclidean plane, but it does not require the space be geodesic.

Let $(X,d)$ be a metric space. A triple of points $\l(a_1,a_2,a_3\r)$ in $X$ is called a \emph{triangle} and the points $a_1,a_2,a_3$ are called its \emph{vertices.} For this triangle in $(X,d),$ there exist points $\ol{a}_1,\ol{a}_2,\ol{a}_3\in\rls^2$ such that
\begin{equation*}
 d\l(a_i,a_j\r)=\l\| \ol{a}_i-\ol{a}_j \r\|,\qquad \text{for every } i,j=1,2,3,
\end{equation*}
where $\|\cdot\|$ stands for the Euclidean distance. The triple of points $\l(\ol{a}_1,\ol{a}_2,\ol{a}_3\r)$ is called a \emph{comparison triangle} for the triangle $\l(a_1,a_2,a_3\r),$ and it is unique up to isometries.

Given these two triangles, we define the functions
\begin{align*}
 \rho_{\l(a_1,a_2,a_3\r)}(x) & =\max_{i=1,2,3} d(x,a_i),\qquad x\in X,\\
\intertext{and,}
 \rho_{\l(\ol{a}_1,\ol{a}_2,\ol{a}_3\r)}(x) & =\max_{i=1,2,3} \l\|x-\ol{a}_i\r\|,\qquad x\in \rls^2.
\end{align*}
The numbers
\begin{equation*}
 r\l(a_1,a_2,a_3\r)\as \inf_{x\in X} \rho_{\l(a_1,a_2,a_3\r)}(x) \quad\text{and}\quad r\l(\ol{a}_1,\ol{a}_2,\ol{a}_3\r)\as \min_{x\in \rls} \rho_{\l(\ol{a}_1,\ol{a}_2,\ol{a}_3\r)}(x)
\end{equation*}
are called the \emph{circumradii} of the respective triangles. Next we can introduce our main definition.
\begin{definition}[Nonpositive curvature] \label{def:ournpc}
 Let $(X,d)$ be a metric space. We say that $\curv X\leq0$ if, for each triangle $\l(a_1,a_2,a_3\r)$ in $X,$ we have
\begin{equation}  \label{eq:def}
 r\l(a_1,a_2,a_3\r)\leq r\l(\ol{a}_1,\ol{a}_2,\ol{a}_3\r),
\end{equation}
where $\ol{a}_i$ with $i=1,2,3$ are the vertices of an associated comparison triangle.
\end{definition}

As we shall see, our definition of nonpositive curvature is implied by
the CAT(0) property, but not by nonpositive curvature in the sense of
Busemann. In Riemannian manifolds, however, all of them are equivalent
to global nonpositive \emph{sectional} curvature. We also make a
connection to the celebrated Kirszbraun extension theorem.

In order to appreciate the geometric content of our definition, let us assume that the infimum in \rf{eq:def} is attained, i.e., there exists some $m\in X$ with 
\bel{eq:2}
d(m,a_i) \le r\l(a_1,a_2,a_3\r)= \inf_{x\in X} \rho_{\l(a_1,a_2,a_3\r)}(x).
\qe
We then call such an $m$ a \emph{circumcenter} of the triangle with vertices $a_1,a_2,a_3.$ This can be equivalently expressed as
\bel{eq:3}
\bigcap_{i=1,2,3} B\l(a_i,r\l(a_1,a_2,a_3\r)\r)\neq \emptyset,
\qe
where $B(x,r)\as\{ y\in X: d(x,y) \le r\}$ denotes a closed distance
ball. The intersection is nonempty because it contains the point
$m$. The condition \rf{eq:3} as such, however, does not involve the
point $m$ explicitly. Our curvature inequality thus embodies the
principle that three balls in $X$ should have a nonempty intersection
whenever the corresponding balls in the Euclidean plane with the same
distances between their centers intersect nontrivially. It therefore is meaningful in a general metric
space to search for the minimal radius for which the balls centered at
three given points have a nonempty intersection. In such a general
context, curvature bounds can therefore be interpreted as
quantification of the dependence of such a minimal radius on the
distances between the points involved, as compared to the Euclidean
situation. Below, we shall also discuss how this principle can be
adapted to discrete metric spaces. This should justify the word
``general'' in the title of our paper. 

It is also worth mentioning that our definition of nonpositive curvature is stable under the Gromov-Hausdorff convergence.


\section{Preliminaries}

We first introduce some terminology from metric geometry and recall a few facts. As references on the subject, we recommend \cite{bacak,BridsonHaefliger99,Jost97}. Let $(X,d)$ be a metric space and let $x,y\in X.$ If there exists a point $m\in X$ such that $d(x,m)=d(m,y)=\frac12d(x,y),$ we call it a \emph{midpoint} of $x,y.$ Similarly, we say that a pair of points $x,y\in X$ has \emph{approximate midpoints} if for every $\eps>0$ there exists $m\in X$ such that
\begin{equation*}
 \max\l\{d(x,m),d(y,m)\r\}\le\frac12d(x,y)+\eps.
\end{equation*}

A continuous mapping $\gamma\col[0,1]\to X$ is called a \emph{path} and its \emph{length} is defined as
\begin{equation*}
\length(\gamma)\as\sup \sum_{i=1}^n d\l(\gamma\l(t_{i-1}\r),\gamma\l(t_i\r) \r), 
\end{equation*}
where the supremum is taken over the set of all partitions $0=t_0<\cdots<t_n=1$ of the interval~$[0,1],$ with an arbitrary $n\in\nat.$ Given $x,y\in X,$ we say that a path $\gamma\col [0,1]\to X$ joins $x$ and $y$ if $\gamma(0)=x$ and $\gamma(1)=y.$
A~metric space $(X,d)$ is a \emph{length space} if
\begin{equation*}
 d(x,y)=\inf\l\{\length(\gamma)\col\text{ path } \gamma\text{ joins } x,y  \r\},
\end{equation*}
for every $x,y\in X.$ A complete metric space is a lenght space if and only if each pair of points has approximate midpoints.

A metric space $(X,d)$ is called \emph{geodesic} if each pair of points $x,y\in X$ is joined by a path $\gamma\col[0,1]\to X$ such that
\begin{equation*}
d\l(\gamma(s),\gamma(t)\r)=d(x,y)\:|s-t|,
\end{equation*}
for every $s,t\in[0,1].$ The path $\gamma$ is then called a \emph{geodesic} and occasionally denoted $[x,y].$ If each pair of points is connected by a \emph{unique} geodesic, we call the space \emph{uniquely geodesic.}
 
Denote by $Z_{t}(x,y)$ the set of $t$-midpoints, i.e. $z\in Z_{t}(x,y)$
iff $td(x,y)=d(x,z)$ and $(1-t)d(x,y)=d(z.y).$ A geodesic space
is called non-branching if for each triple of points $x,y,y'\in X$
with $d(x,y)=d(x,y)$ the condition $Z_{t}(x,y)\cap Z_{t}(x,y')\ne\varnothing$ 
for some $t\in(0,1)$ implies that $y=y'$. 
 
\subsection{Busemann spaces}

A geodesic space $(X,d)$ is a Busemann space if and only if, for every geodesics $\gamma,\eta\col[0,1]\to X,$ the function $t\mapsto d \l(\gamma(t),\eta(t)\r)$ is convex on $[0,1].$ This property in particular implies that Busemann spaces are uniquely geodesic.

\subsection{Hadamard spaces}

Let $(X,d)$ be a geodesic space. If for each point $z\in X,$ each geodesic $\gamma\col[0,1]\to X,$ and $t\in[0,1],$ we have
\begin{equation} \label{eq:cat}
d\l(z,\gamma(t)\r)^2\leq (1-t) d\l(z,\gamma(0)\r)^2+td\l(z,\gamma(1)\r)^2-t(1-t) d\l(\gamma(0),\gamma(1)\r)^2,
\end{equation}
the space $(X,d)$ is called CAT(0). It is easy to see that CAT(0) spaces are Busemann. A complete CAT(0) space is called an \emph{Hadamard space.}


\section{Connections to other definitions of NPC} \label{sec:otherNPC}

In this section we study the relationship between Definition~\ref{def:ournpc} and other notions of nonpositive curvature that are known in comparison geometry.

We begin with a simple observation. If a metric space $(X,d)$ is complete and $\curv X\leq0,$ then it is a length space. If we moreover required the function $\rho_{\l(a_1,a_2,a_3\r)}(\cdot)$ from Definition \ref{def:ournpc} to attain its minimum, we would obtain a geodesic space. That is a motivation for introducing scaled and relaxed versions of Definition~\ref{def:ournpc} in Section~\ref{sec:scaled}.

To show that Hadamard spaces have nonpositive curvature in the sense of Definition~\ref{def:ournpc}, we need the following version of the Kirszbraun extension theorem with Hadamard space target. By a nonexpansive mapping we mean a $1$-Lipschitz mapping.
\begin{theorem}[Lang-Schroeder] \label{thm:kirszbraun}
 Let $(\hs,d)$ be an Hadamard space and let $S\subset\rls^2$ be an arbitrary set. Then for each nonexpansive mapping $f\col S\to\hs,$ there exists a nonexpansive mapping $F\col\rls^2\to\hs$ such that $F\res_S=f.$
\end{theorem}
\begin{proof}
 See~\cite{langschroeder}.
\end{proof}

\begin{corollary}
 Let $(\hs,d)$ be an Hadamard space. Then $\curv\hs\leq0.$
\end{corollary}
\begin{proof}
 Consider a triangle with vertices $a_1,a_2,a_3\in\hs$ and apply Theorem~\ref{thm:kirszbraun} to the set
 $S\as\l\{\ol{a}_1,\ol{a}_2,\ol{a}_3\r\}$ and isometry $f\col
 \ol{a}_i\mapsto a_i,$ for $i=1,2,3.$ We obtain a nonexpansive mapping
 $F\col\rls^2\to\hs$ which maps the circumcenter $\ol{a}$ of $S$ to some $a \in \hs$. By nonexpansiveness, $d(a,a_i)\le \| \ol{a}- \ol{a}_i\|$, 
which is exactly the condition in~\eqref{eq:cat}.
\end{proof}
In case of Hadamard manifolds, one can argue in a more elementary way than via Theorem~\ref{thm:kirszbraun}. We include the proof since it is of independent interest. The following fact, which holds in all Hadamard spaces, will be used.
\begin{lemma} \label{lem:varfor}
Let $(\hs,d)$ be an Hadamard space. Assume $\gamma\col [0,1]\to\hs$ is a
geodesic and $z\in\hs\setminus\gamma.$ Then
\begin{equation*}
 \lim_{t\to0+}\frac{d\l(z,\gamma_0\r)-d\l(z,\gamma_t\r)}{t}=\angle\l(\gamma(1),\gamma(0),z\r).
\end{equation*}
The existence of the limit is part of the statement. The RHS denotes
the angle at $\gamma(0)$ between $\gamma$ and $\l[\gamma(0),z\r].$
\end{lemma}
\begin{proof}
 Cf. \cite[p. 185]{BridsonHaefliger99}.
\end{proof}
\begin{theorem}
 Let $M$ be an Hadamard manifold. Then $\curv M\leq0.$
\end{theorem}
\begin{proof}
 Choose a triangle with vertices $a_1,a_2,a_3\in M$ and observe that the set of minimizers of the function $\rho_{\l(a_1,a_2,a_3\r)}(\cdot)$ coincides with the set of minimizers of the function
\begin{equation*}
 x\mapsto \max_{i=1,2,3} d(x,a_i)^2,\qquad x\in X.
\end{equation*}
Since the latter function is strongly convex on Hadamard manifolds, it has a unique minimizer $m\in M.$ Denote $b_i=\exp_m^{-1}(a_i)\in T_m M$ for every $i=1,2,3.$ We claim that $b_1,b_2,b_3$ lie in a plane containing $0.$ Indeed, if it were not the case, there would exist a vector $v\in T_m M$ such that $\l\langle v,b_i\r\rangle_m >0$ for every $=1,2,3.$ According to Lemma \ref{lem:varfor} we would than have
\begin{equation*}
 \lim_{t\to0+}\frac{d\left(b_i,m \right)-d\left(b_i,\exp_m(tv) \right)}{t}>0,
\end{equation*}
for every $i=1,2,3.$ There is hence $\eps>0$ such that
\begin{equation*}
 d\l(b_i,m \r)>d\l(b_i,\exp_m(tv) \r),
\end{equation*}
for every $t\in(0,\eps)$ and $i=1,2,3.$ This is a contradiction
to $m$ being a minimizer of $\rho_{\l(a_1,a_2,a_3\r)}(\cdot).$ We can therefore
conclude that $b_1,b_2,b_3$ lie in a plane containing $0.$

By an elementary Euclidean geometry argument we obtain that
\begin{equation*}
 r\l(a_1,a_2,a_3\r)= r\l(b_1,b_2,b_3\r).
\end{equation*}
Since $\sec M\leq0,$ we have $\|b_i-b_j\|\leq d(a_i,a_j)$ for every
$i,j=1,2,3.$ Consequently,
\begin{equation*}
 r\l(\ol{a}_1,\ol{a}_2,\ol{a}_3 \r)\geq r\l(b_1,b_2,b_3\r),
\end{equation*}
which finishes the proof.
\end{proof}
The converse implication holds as well.
\begin{theorem}\label{th:CurvimpliesSec}
Let $M$ be a smooth manifold with $\curv M\leq0$. Then $\sec M\leq0$.
\end{theorem}
We shall prove this theorem in the remainder of the present section. Naturally, our arguments are local, and to obtain nonpositive sectional curvature at a point $m\in M,$ we only need assume that \eqref{eq:def} is satisfied for all sufficiently small triangles around this point.

Let us choose a plane $\Pi \subset T_m M$ and pick three unit vectors $X,Y,Z \in\Pi$ with
\begin{equation}\label{eq:anglesaddup}
 \angle XOY=\angle YOZ=\angle ZOX=\frac{2\pi}{3},
\end{equation}
where $O \in T_m M$ is the origin and the angles are measured in the metric of $T_m M$. Furthermore, we set $\gamma_X(t)\as\exp_mtX,\gamma_Y(t)\as\exp_mtY,\gamma_Z(t)\as\exp_m tZ,$ for all small $t>0.$  

\begin{lemma} \label{lem:aux}
For sufficiently small $t,$ the center of the minimal enclosing ball of $\gamma_X(t),\gamma_Y(t),\gamma_Z(t)$ in $M$ is $m,$ and the circumradius is equal to $t,$ that is,
\begin{equation*}
 t = \rho_{\gamma_X(t),\gamma_Y(t),\gamma_Z(t)}(m) = r\l(\gamma_X(t),\gamma_Y(t),\gamma_Z(t)\r).
\end{equation*}
\end{lemma}
\begin{proof}
For small enough $t$ we can pick a convex neighborhood $W\subset M$ of $m$ containing $\gamma_X(t),\gamma_Y(t),\gamma_Z(t)$ such that the $\rho_{\gamma_X(t),\gamma_Y(t),\gamma_Z(t)}(\cdot)$ has a unique minimizer on $W.$

Fix a unit vector $V \in T_m M$. We first show that there exists $U \in \{X,Y,Z\}$ such that
\begin{equation}\label{eq:onedirectionisfine}
 \lim_{\eps \to 0+} \frac{d (\gamma_U(t),\exp_m(\eps V))^2-d (\gamma_U(t),m)^2}{\eps} \geq 0,
\end{equation}
and $\rho(m) = d \l(\gamma_U(t),m\r).$ If that were true, then for all $z\in W$ sufficiently close to $m$, we would have
\begin{equation*}
 \frac{\rho(\exp_m(\eps V))^2-\rho(m)^2}{\eps} \geq \frac{d (\gamma_U(t),\exp_m(\eps V))^2-d (\gamma_U(t),m)^2}{\eps}.
\end{equation*}
Hence
\begin{equation*}
 \lim_{\eps \to 0+} \frac{\rho(\exp_m(\eps V))^2-\rho(m)^2}{\eps} \geq 0.
\end{equation*}
If this holds for any $V \in T_m M$, together with the convexity of $\rho$ we obtain that $m$ is the unique minimizer.	

To prove the existence of $U$ satisfying \eqref{eq:onedirectionisfine}, decompose $V = \lambda V^\Pi + \mu V^\perp$, where $V^\Pi \in \Pi$ and $V^\perp \in \Pi^\perp$ are unit vectors and $\lambda,\mu\in\rls.$
On the one hand, for any $U \in \Pi$ by the first variation of the distance function gives
\begin{equation*}
 \frac{\di\;\;}{\di\eps} \Big |_{\eps = 0} d(\gamma_{U}(t),\exp_m(\eps V^\perp))^2 = 0.
\end{equation*}
On the other hand, by \eqref{eq:anglesaddup} there has to be some $U \in \{X,Y,Z\}$ such that
\begin{equation*}
 \angle UOV^\Pi  \geq \frac{2\pi}{3} > \frac{\pi}{2}.
\end{equation*}
If $W$ is small enough (independent of $t$), the uniform bound away from $\frac{\pi}{2}$ implies that for any sufficiently small $\eps > 0$
\begin{equation*}
 d (\gamma_{U}(t),\exp_m(\eps V^\Pi)) > d (\gamma_{U}(t),m),
\end{equation*}
which establishes \eqref{eq:onedirectionisfine}. Since we have $\rho(m) = d \l(\gamma_U(t),m\r),$ the Lemma \ref{lem:aux} is proved.
\end{proof}

Now let $\ol{x}(t),\ol{y}(t),\ol{z}(t)$ be a comparison triangle in $\R^2$ for the geodesic triangle $\gamma_X(t),\gamma_Y(t),\gamma_Z(t)$ in $M.$ Let $\ol{m}(t)$ be the minimizer of $\rho_{\ol{x}(t),\ol{y}(t),\ol{z}(t)}(\cdot)$ in $\R^2$. By Lemma~\ref{lem:aux} and by the assumption of nonpositive curvature in the sense of Definition~\ref{def:ournpc}, we have
\begin{equation*}
 t =r\l(\gamma_X(t),\gamma_Y(t),\gamma_Z(t)\r)\leq  r\l(\ol{x}(t),\ol{y}(t),\ol{z}(t)\r).
\end{equation*}
On the other hand, the origin $O$ is the minimizer of $\rho_{(tX,tY,tZ)}(\cdot).$ Now we can conclude from a fully Euclidean argument the following. There are two possibilities for $\ol{m}(t).$ It either lies on one of the sides of the triangle $\ol{x}(t),\ol{y}(t),\ol{z}(t),$ say $\ol{m}(t) \in \l[\ol{x}(t),\ol{y}(t)\r],$ in which case
\begin{equation*}
 \l\|\ol{x}(t)-\ol{y}(t)\r\| \ge 2t > \l\|tX-tY\r\|,
\end{equation*}
or $\ol{m}$ has equal distance to $\ol{x}(t),\; \ol{y}(t),$ and $\ol{z}(t),$ so at least one angle at $\ol{m}$ is greater or equal $\frac{2\pi}3,$ say $\angle \l(\ol{x}(t),\ol{m}(t),\ol{y}(t)\r) \geq \frac{2\pi}{3}.$ Since the angle $\angle X O Y = \frac{2\pi}{3}$, it must be that
\begin{equation*}
 \l\|\ol{x}(t)-\ol{y}(t)\r\| \geq  \l\|tX-tY\r\|.
\end{equation*}
Given $t>0,$ there exist therefore $U,V \in \{X,Y,Z\},$ with $U\neq V,$ such that
\begin{equation*}
 d(\gamma_U(t),\gamma_V(t)) = \l\|\ol{u}(t)-\ol{v}(t) \r\| \geq  \l\|tU-tV\r\|,
\end{equation*}
and in particular there exists a sequence $t_i \to 0$ and $U,V \in \{X,Y,Z\},$ with $U\neq V,$ such that this holds for any $i \in \N$. By the following Lemma \ref{t:sectional curvature}, the sectional curvature of the plane $\Pi\as\spn(X,Y,Z)$ at $m$ is nonpositive. Its proof follows from the second variation formula of the energy of geodesics.

\begin{lemma}\label{t:sectional curvature}
Let $X,Y\in T_m M$ be two independent unit tangent vectors at $m$. 
Then \begin{equation*}
\lim_{t\to 0}\frac{1}{t^2}\left(\frac{d(\exp_m(tX),\exp_m(tY))}{t\l\|X-Y\r\|}-1\right)=-C\l(n,\angle(X,Y)\r) K(X,Y),
\end{equation*}
where $K(X,Y)$ is the sectional curvature of $\spn(X,Y).$ In particular, if the sectional curvature of a plane $\Pi \subset T_m M$ is finite, and there exist unit vectors $X$ and $Y$ spanning $\Pi$ such that for some sequence $t_i \to 0,$
\begin{equation*}
 d\l(\exp_m(t_iX),\exp_m(t_iY)\r) \geq t_i\l\|X-Y\r\|, \quad \text{for each } i \in \nat,
\end{equation*}
then the sectional curvature of $\Pi$ is nonpositive.
\end{lemma} 

The proof of Theorem \ref{th:CurvimpliesSec} is now complete.


\section{Scaled and relaxed nonpositive curvature} \label{sec:scaled}

We now introduce a quantitative version of nonpositive curvature from Definition~\ref{def:ournpc}. It is appropriate in discrete metric spaces.
\begin{definition} \label{def:scalednpc}
A metric space $(X,d)$ has nonpositive curvature at scale $\beta>0$ if
\begin{equation*}
 \inf_{x\in X}\rho_{\l(a_1,a_2,a_3\r)}(x)\leq \min_{x\in\rls^2}  \rho_{\l(\ol{a}_1,\ol{a}_2,\ol{a}_3\r)}(x),
\end{equation*}
for every triangle $a_1,a_2,a_3\in X$ such that $d\l(a_i,a_j\r)\geq\beta$ for every $i,j=1,2,3$ with $i\neq j.$ Denote this curvature condition by $\curv_\beta X\leq0.$
\end{definition}
Again, whenever the infimum is attained, this condition can be formulated in terms of intersections of distance balls.

Another way to relax the nonpositive curvature condition from Definition~\ref{def:ournpc} is to allow a small error.
\begin{definition} \label{def:relaxednpc}
A metric space $(X,d)$ has $\eps$-relaxed nonpositive curvature, where $\eps>0,$ if
\begin{equation*}
 \inf_{x\in X}\rho_{\l(a_1,a_2,a_3\r)}(x)\leq \min_{x\in\rls^2}  \rho_{\l(\ol{a}_1,\ol{a}_2,\ol{a}_3\r)}(x)+\eps,
\end{equation*}
for every triangle $a_1,a_2,a_3\in X.$ Denote this curvature condition by $\eps\textrm{-}\curv X\leq0.$
\end{definition}
We will now observe that this relaxed nonpositive curvature is enjoyed by $\delta$-hyperbolic spaces, where $\delta>0.$ Recall that a geodesic space is $\delta$-hyperbolic if every geodesic triangle is contained in the $\delta$-neighborhood of its arbitrary two sides \cite[p. 399]{BridsonHaefliger99}. Consider thus a geodesic triangle $a_1,a_2,a_3\in X$ in a $\delta$-hyperbolic space $(X,d).$ By the triangle inequality, one can see that
\begin{equation*}
 \inf_{x\in X}\rho_{\l(a_1,a_2,a_3\r)}(x)\leq \frac12\max_{i,j=1,2,3}  d\l(a_i,a_j\r)+2\delta \leq \min_{x\in\rls^2}  \rho_{\l(\ol{a}_1,\ol{a}_2,\ol{a}_3\r)}(x)+2\delta,
\end{equation*}
and therefore $2\delta\textrm{-}\curv X\leq0.$

This in particular applies to Gromov hyperbolic groups. A group is called \emph{hyperbolic} if there exists $\delta>0$ such that its Cayley graph is a $\delta$-hyperbolic space. The above discussion hence implies that a hyperbolic group has $\eps$-relaxed nonpositive curvature for some $\eps>0.$ We should like to mention that Y.~Ollivier has recently esthablished coarse Ricci curvature for hyperbolic groups \cite[Example 15]{ollivier}. For more details on hyperbolic spaces and groups, the reader is referred to \cite{BridsonHaefliger99}.

In conclusion, Definitions \ref{def:scalednpc} and \ref{def:relaxednpc} require ``large'' triangles only to satisfy some nonpositive curvature conditions, whereas ``small'' triangles can be arbitrary. This, in particular, allows for the notion of nonpositive curvature in discrete metric spaces and might be useful in \emph{geometric data analysis.}


\section{From local to global}

In the sense of Alexandrov, any simply-connected geodesic space with local nonpositive 
curvature has global nonpositive curvature, i.e. CAT(0): It is a natural question to 
ask when such kind of globalization theorem holds for our curvature definition.

\begin{definition} \label{def:localnpc}
A metric space $(X,d)$ has local nonpositive curvature if for each $x\in X$ there is a neighborhood
$U$ such that the curvature condition holds for all triangles in $U$. We denote this 
curvature condition by $\curv_{loc} X \le 0$. 
\end{definition}
If every point $x$ admits a convex neighborhood $U_x$ then the condition can be also written as
$$
\curv_{loc} X \le 0 \Longleftrightarrow \forall x\in X: \curv U_x \le 0.
$$  

\begin{theorem}Assume that $(X,d)$ is a geodesic space with $\curv_{loc} X\leq0$ and the circumcenter is attained for every triangle $\{a_i\}_{i=1}^3.$ If $(X,d)$ is globally nonpositive curved in the sense of Busemann, then we have $\curv X\leq 0.$
\end{theorem}
\begin{proof}For any triangle $\{a_i\}_{i=1}^3$ in $X,$ it suffices to show \eqref{eq:def}. 
Let $m$ be the circumcenter of the triangle $\{a_i\}_{i=1}^3$ and $ma_i, 1\leq i\leq3,$ the 
minimizing geodesic connecting $m$ and $a_i.$ For any $t>0,$ let $a_i^t$ be the point on the 
geodesic $m a_i$ such that $|ma_i^t|=t|ma_i|.$ By the contradiction argument, one can show 
that $m$ is the circumcenter of the triangle $\{a_i^t\}_{i=1}^3$ for any $t>0,$ and hence 
$r\l(a_1^t,a_2^t,a_3^t\r)=t\cdot r\l(a_1,a_2,a_3\r).$ 

By the local curvature condition, there exists a small neighbourhood $U_m$ of $m$ in which 
the comparison \eqref{eq:def} holds. We know that for sufficiently small $t>0,$ $a_{i}^t\in U_m$ 
for $1\leq i\leq 3.$ Hence for the corresponding comparison triangle $\{\ol{a}_i^t\}$ in $\R^2,$ 
we have
$$ r\l(a_1^t,a_2^t,a_3^t\r)\leq r\l(\ol{a}_1^t,\ol{a}_2^t,\ol{a}_3^t\r).$$ 
The global Busemann condition for the triangle $\{m,a_i,a_{i+1}\},1\leq i\leq 3,$ implies 
that $|a_i^ta_{i+1}^t|\leq t|a_i,a_{i+1}|$ where the indices are understood in the sense of 
module $3.$ Hence  $r\l(\ol{a}_1^t,\ol{a}_2^t,\ol{a}_3^t\r)\leq  t\cdot r\l(\ol{a}_1,\ol{a}_2,\ol{a}_3\r).$ 

Combining all these facts, we have $r\l(a_1,a_2,a_3\r)\leq r\l(\ol{a}_1,\ol{a}_2,\ol{a}_3\r)$ 
which proves the theorem.
\end{proof}


\section{The Kirszbraun theorem and general curvature bounds}
In this section, we shall describe that our constructions and results can be extended to curvature bounds other than $0.$ We shall generalize the result of Section~\ref{sec:otherNPC} to arbitrary curvature bounds both from above and below. For that purpose, we show a direct implication of our curvature comparison by Kirszbraun's theorem on Lipschitz extension by \cite{langschroeder}. Given $\kappa\in\rls,$ let $\CBB(\kappa)$ denote the class of Alexandrov spaces with sectional curvature bounded from below by $\kappa$ and $\CAT(\kappa)$ the class of spaces with sectional curvature bounded from above by $\kappa.$ As a reference on Alexandrov geometry, we recommend \cite{BridsonHaefliger99,bbi}. The symbol $\l(\model,d_\kappa\r)$ stands for the model plane, as usually.

\begin{theorem}[Kirszbraun's theorem] \label{thm:kirszbraunk}
Let $\mathcal{L}\in \CBB(\kappa),$ $\mathcal{U}\in \CAT(\kappa),$
$Q\subset \mathcal{L}$ and $f\col Q\to \mathcal{U}$ be a nonexpansive map.
Assume that there is $z\in \mathcal{U}$ such that $f(Q)\subset
B\l(z,\frac{\pi}{2\sqrt{\kappa}}\r)$ if $\kappa>0.$ Then $f\col Q\to
\mathcal{U}$ can be extended to a nonexpansive map $F\col \mathcal{L}\to
\mathcal{U}.$\end{theorem}
\begin{proof}
 Cf. \cite{langschroeder}.
\end{proof}

As a convention, we exclude large triangles in what follows if $\kappa>0.$
\begin{definition}
 Let $(X,d)$ be a metric space. We say that $\curv X\leq\kappa$ if, for each triangle $\l(a_1,a_2,a_3\r)$ in $X,$ we have $r\l(a_1,a_2,a_3\r)\leq r\l(\ol{a}_1,\ol{a}_2,\ol{a}_3\r),$ where $\ol{a}_i$ with $i=1,2,3$ are the vertices of an associated comparison triangle in $\model.$ Similarly, we say that $\curv X\geq\kappa$ if, for each triangle $\l(a_1,a_2,a_3\r)$ in $X,$ we have $r\l(a_1,a_2,a_3\r)\geq r\l(\ol{a}_1,\ol{a}_2,\ol{a}_3\r),$ where $\ol{a}_i$ with $i=1,2,3$ are the vertices of an associated comparison triangle in $\model.$
\end{definition}

\begin{theorem}\label{t:short proof by Kirszbraun}
Let $(X,d)$ be a $\CAT(\kappa)$ space. Then $\curv X \le \kappa.$
\end{theorem}
\begin{proof}
Let $\l(x_1,x_2,x_3\r)$ be a triangle in $X$ and $\l(\ol{x}_1,\ol{x}_2,\ol{x}_3\r)$ be the
comparison triangle in $\model.$ By the definition of comparison
triangle, the map $f\col \l\{\ol{x}_1,\ol{x}_2,\ol{x}_3\r\}\to X$ defined by $f\l(\ol{x}_i\r)=x_i$
for $i=1,2,3$ is an isometry. By Theorem \ref{thm:kirszbraunk}, the mapping $f$ can be extended to a nonexpansive map $F\col\model\to X.$ Let $\ol{m}\in\model$ be the circumcenter of $\ol{x}_1,\ol{x}_2,\ol{x}_3.$ Then by the nonexpansiveness of the map $F,$ we have
\begin{equation*}
d\l(F\l(\ol{m}\r),x_i\r)\leq d_\kappa\l(\ol{m},\ol{x}_i\r)\leq r\l(\ol{x}_1,\ol{x}_2,\ol{x}_3\r),\qquad i=1,2,3.
\end{equation*}
Hence, we have $r\l(x_1,x_2,x_3\r)\leq r\l(\ol{x}_1,\ol{x}_2,\ol{x}_3\r)$ by the very definition of the circumradius of $\l(x_1,x_2,x_3\r).$
\end{proof}

We note that we can also define a lower curvature bound for any $\kappa$
by requiring that the circumradius of comparison triangle for a triangle
$\Delta$ is less than or equal to the circumcenter of the triangle $\Delta$. 
Similar to the theorem above one can use Kirszbraun's theorem to prove:
\begin{theorem}\label{t:short proof by Kirszbraun lower}
Let $(X,d)$ be a $\CBB(\kappa)$ space. Then $\curv X \ge \kappa.$
\end{theorem}
\begin{proof}
Let $\l(x_1,x_2,x_3\r)$ be a triangle in $X$ and $\l(\ol{x}_1,\ol{x}_2,\ol{x}_3\r)$ be the
comparison triangle in $\model.$ By the definition of comparison
triangle, the map $f\col \{x_1,x_2,x_3\}\to\model$ defined by $f(x_i)=\ol{x}_i$
for $i=1,2,3$ is an isometry. By Kirszbraun's theorem, the map $f$ can be extended to a nonexpansive map $F\col X\to\model.$ Let $\l(m_l\r)\subset X$
be a minimizing sequence of the function $\rho_{x_1,x_2,x_3}(\cdot).$ Then by the nonexpansiveness of the map $F,$ we have
\begin{equation*}
\lim_{l\to\infty} d_\kappa\l(F\l(m_l\r),\ol{x}_i\r)\leq \lim_{l\to\infty} d\l(m_l,x_i\r) \leq r\l(x_1,x_2,x_3\r),\ \ i=1,2,3.\end{equation*}
Hence, we have $r\l(\ol{x}_1,\ol{x}_2,\ol{x}_3\r)\leq r\l(x_1,x_2,x_3\r)$ by the very definition of the circumradius of $\l(x_1,x_2,x_3\r).$
\end{proof} 

We end this section by showing a metric implication, which generalized an observation made in Section~\ref{sec:otherNPC}.
\begin{proposition}
Let $(X,d)$ be a complete metric space. If $\curv X\le\kappa$ for some $\kappa\in\rls,$ then it is a length space.
\end{proposition}
\begin{proof}
We will show that each pair of points has approximate midpoints. Let $x,y\in X$ and choose a triangle $\l(x_1,x_2,x_3\r)$ in $X$ such that $x_1=x_2=x$ and $x_3=y$. The circumcenter of the comparison triangle is the midpoint $\ol{m}$ of the geodesic $\ol{x}_1,\ol{x}_3$.
By our assumptions we have the inequality
\begin{equation*}
 r\l(x_1, x_2, x_3 \r) \le r\l(\ol{x}_1, \ol{x}_2, \ol{x}_3\r)=\frac12 \l\| \ol{x}_1-\ol{x}_3\r\|=\frac12d(x,y),
\end{equation*}
and since it always holds $\frac12d(x,y)\le r\l(x_1, x_2, x_3\r),$ we obtain that there exists a sequence $\l(m_l\r)\subset X$ such that $d(x,m_l),d(y,m_l)\to \frac12d(x,y).$ That is, the pair of points $x,y$ has approximate midpoints.
\end{proof}

\begin{proposition}
Let $(X,d)$ be a geodesic space. If $\ensuremath{\curv X\ge\kappa,}$
then the space is non-branching.\end{proposition}
\begin{proof}
Assume the space is branching, then there are three distinct point
$x,y,y'\in X$ such that $z\in Z_{\frac{1}{2}}(x,y)\cap Z_{\frac{1}{2}}(x,y')$
and it is not difficult to see that $d(y,y')\le d(x,y)=d(x,y')$ and
thus $r(x,y,y')=d(x,y)/2$. Note, however, that the corresponding
comparison triangle $(\bar{x},\bar{y},\bar{y}')$ is a regular isosceles
triangle and hence $r(\bar{x},\bar{y},\bar{y}')>d(\bar{x},\bar{y})/2=r(x,y,y')$.
But this violates the curvature conditions and hence the space cannot
contain branching geodesics. 
\end{proof}


\section{$L^p$-spaces and the curvature condition}
We have the following surprising result for the new curvature condition.
\begin{theorem}[Curvature of $L^p$-spaces]
We have $\curv L^p\leq 0$ if and only if $p=2$ or $p=\infty.$
\end{theorem}
\begin{proof}
Since $\curv L^p\leq 0$ trivially, we show that $\curv L^\infty\leq 0.$ Let $\l(x_1,x_2,x_3\r)$ be a triangle in $L^\infty$ 
and without loss of generality $\l[x_1,x_2\r]$ be the longest side. Furthermore let $\l(\ol{x}_1,\ol{x}_2,\ol{x}_3\r)$ 
be the comparison triangle in $\R^2$ for the triangle $\l(x_1,x_2,x_3\r).$ We claim that 
\begin{equation*}
r\l(x_1,x_2,x_3\r)=\frac12 \l\|x_1-x_2\r\|_\infty.
\end{equation*}
Because in $\R^2$ one always has $\frac12\l\|\ol{x}_1-\ol{x}_2\r\|\le r\l(\ol{x}_1,\ol{x}_2,\ol{x}_3\r)$, we have 
$r \l(x_1,x_2,x_3 \r)\le r\l(\ol{x}_1,\ol{x}_2,\ol{x}_3\r)$ and thus $\curv L^\infty \leq0.$

In order to prove the claim, we will construct a circumcenter explicitly: Let 
$c\in L^\infty$ be the point that is for each coordinate a circumcenter, that is,  
for coordinate $l\in \nat$ if $x_i^l\le x_j^l \le x_k^l$ then 
$c^l=\frac{x_k^l-x_i^l}{2}$. One can easily see that $\l\|c-x_i\r\|_\infty=\frac12\l\|x_1-x_2\r\|_\infty$ and that 
there cannot be any point closer to all three points at once.

We now turn to the statement about $L^p$-spaces with $p\in(1,2)\cup(2,\infty).$ In order to show that these $L^p$-spaces do not satisfy the curvature 
condition we will construct explicit counterexamples. For this note that it suffices to show that $\l(\R^2,\|\cdot\|_p\r)$ admits a counterexample, since each $L^p$-space with $p\in(1,\infty)$ contains $\l(\R^2,\|\cdot\|_p\r)$.

Assume first that $2<p<\infty$ and let $(A,B,C)$ be the triangle with coordinates 
$A=(0,1),B=(-1,0)$ and $C=(1,0)$, see Figure \ref{figp1}. The length of the sides with respect to the $L^p$-norm
are $a=2$ and $b=c=\sqrt[p]{2}<\sqrt{2}$.
Now find the triangle $A'=(0,y)$ such that $y>1$ and $b'=c'=\sqrt{2}$. One easily sees
that this triangle is not obtuse and thus $r\l(A',B,C\r)>1$, but the corresponding comparison 
triangle in $\R^2$ is rectangular and its circumradius is $1$. All $L^p$-spaces with 
$2<p<\infty$ do not satisfy the curvature condition.
\begin{figure}
\centering
\begin{minipage}{.5\textwidth}
  \centering
  \includegraphics[width=.8\linewidth]{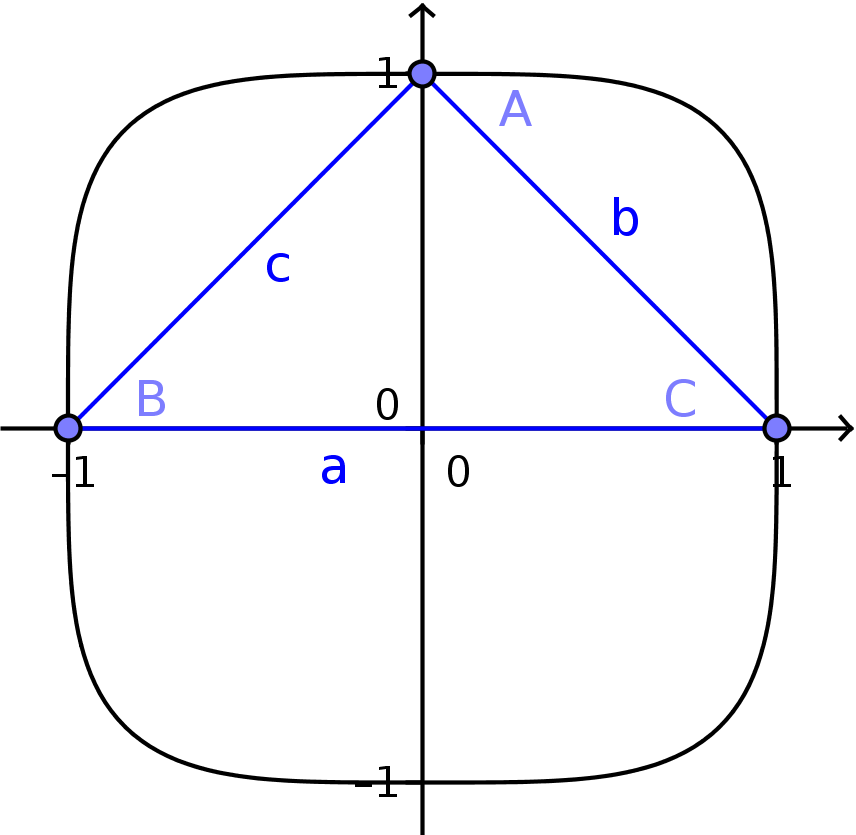}
  \caption{}{$2<p<\infty$}
  \label{figp1}
\end{minipage}%
\begin{minipage}{.5\textwidth}
  \centering
  \includegraphics[width=.8\linewidth]{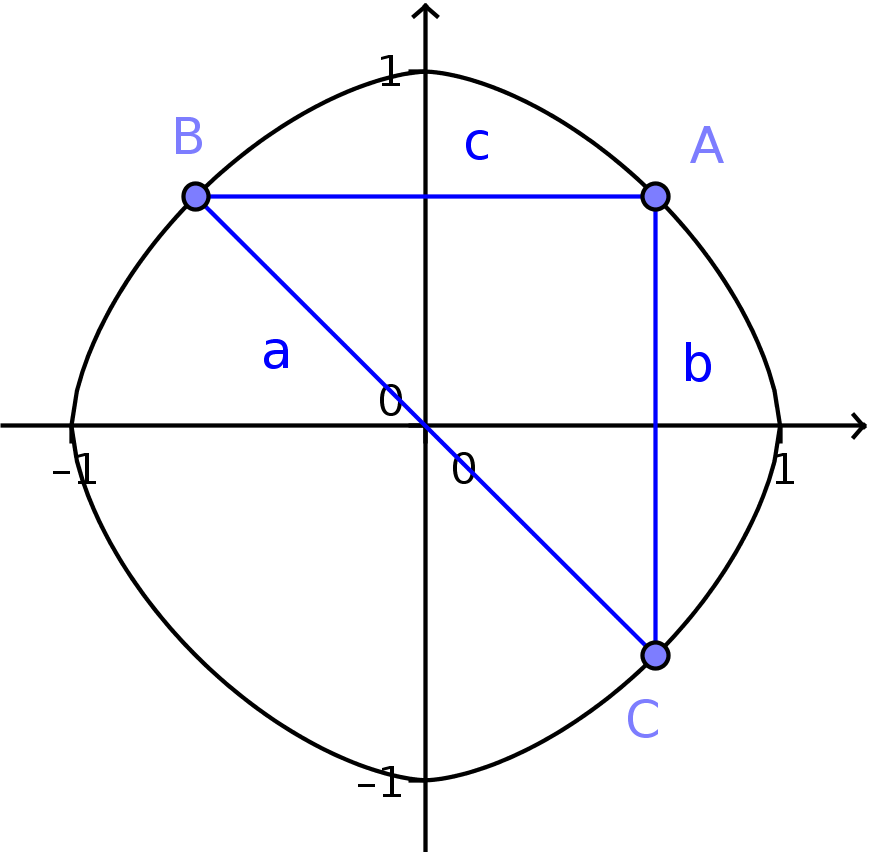}
  \caption{}{$1<p<2$}
  \label{figq2}
\end{minipage}
\end{figure}

Now assume $1<p<2$. We assume again that $\l(A,B,C\r)$ is a triangle on the 
$L^p$-unit sphere with coordinate $A=(r,r), B=(-r,r)$ 
and $C=(r,-r)$, see Figure \ref{figq2}. One easily see that $r=\frac{1}{\sqrt[p]{2}}$ and that 
\begin{equation*}
b=c=2r=\frac{2}{\sqrt[p]{2}} < \frac{2}{\sqrt{2}} = \sqrt{2}.
\end{equation*}
Thus we can again find a point $A'=(r',r')$ with $r'>r$ and $b'=c'=\sqrt{2}$. This 
triangle is not obtuse and $r\l(A',B,C\r)>1$. Hence $L^p$ with $1<p<2$ does not satisfy 
the curvature condition.
\end{proof}

\begin{remark}
Actually it is not difficult to show that $L^p$-spaces with $p\in(1,2)\cup(2,\infty)$ do not even satisfy a lower curvature bound. 
\end{remark}
\begin{proof}
In order to show that no $L^p$-space except for $L^2$ can satisfy a lower curvature bound take 
the two triangle above but change the condition $2<p<\infty$ and $1<p<2$, see Figure \ref{figp2} 
and \ref{figq1}. Now point $A'$ will lie inside the unit sphere and the corresponding triangles 
$A'BC$ are in the interior of obtuse triangle with respect to the $L^p$-norm. Since the comparison 
triangle in $\R^2$ is rectangular we can create a acute isosceles triangle $\tilde{\Delta}$ 
with base side length $1$. Since the triangle $A'BC$ is in the interior of obtuse triangles the 
triangle $\Delta$ corresponding to the comparison triangle $\tilde{\Delta}$ will be obtuse as 
well and its circumradius is $1$. Since $\tilde{\Delta}$ is regular, we see that its circumradius 
is greater than $1$, hence $\Delta$ is a counterexample to a lower curvature bound.
\end{proof}
\begin{figure}
\centering
\begin{minipage}{.5\textwidth}
  \centering
  \includegraphics[width=.8\linewidth]{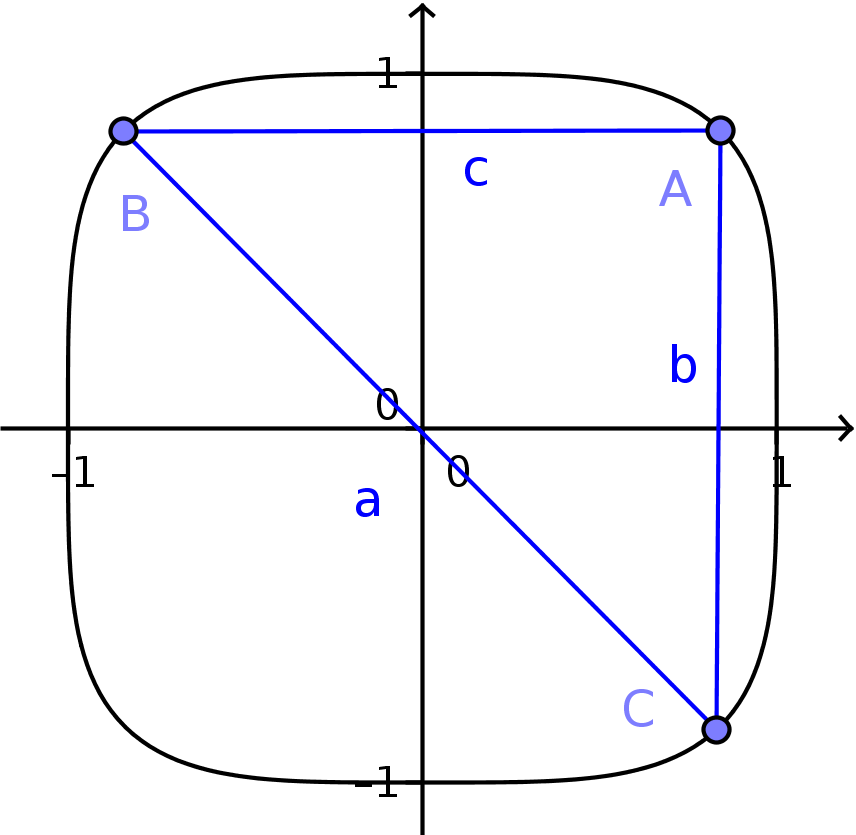}
  \caption{}{$2<p<\infty$}
  \label{figp2}
\end{minipage}%
\begin{minipage}{.5\textwidth}
  \centering
  \includegraphics[width=.8\linewidth]{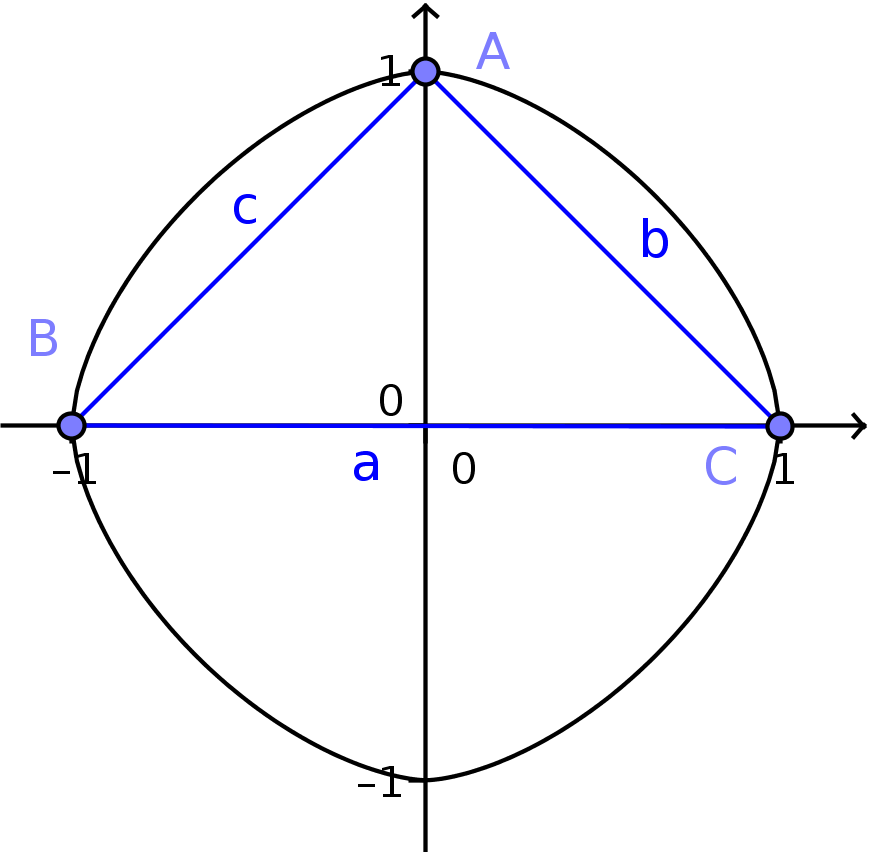}
  \caption{}{$1<p<2$}
  \label{figq1}
\end{minipage}
\end{figure}



\bibliography{npc}

\begin{thebibliography}{1}

\bibitem{bacak}
{\sc M.~Ba\v{c}\'ak}, {\em Convex analysis and optimization in Hadamard
  spaces}, De Gruyter Series in Nonlinear Analysis and Applications, Walter de
  Gruyter \& Co., Berlin. To Appear.

\bibitem{BridsonHaefliger99}
{\sc M.~R. Bridson and A.~Haefliger}, {\em Metric spaces of non-positive
  curvature}, no.~319 in Grundlehren der Mathematischen Wissenschaften,
  Springer-Verlag, Berlin, 1999.

\bibitem{bbi}
{\sc D.~Burago, Y.~Burago, and S.~Ivanov}, {\em A course in metric geometry},
  vol.~33 of Graduate Studies in Mathematics, American Mathematical Society,
  Providence, RI, 2001.

\bibitem{Jost97}
{\sc J.~Jost}, {\em Nonpositive curvature: geometric and analytic aspects},
  Lectures in Mathematics ETH Z\"urich, Birkh\"auser Verlag, Basel, 1997.

\bibitem{langschroeder}
{\sc U.~Lang and V.~Schroeder}, {\em Kirszbraun's theorem and metric spaces of
  bounded curvature}, Geom. Funct. Anal., 7 (1997), pp.~535--560.

\bibitem{ollivier}
{\sc Y.~Ollivier}, {\em Ricci curvature of {M}arkov chains on metric spaces},
  J. Funct. Anal., 256 (2009), pp.~810--864.

\end{thebibliography}
\bibliographystyle{siam}

\end{document}